\documentclass[12pt]{amsart}
\usepackage{graphicx}
\usepackage{commath}
\usepackage{stmaryrd}
\usepackage{xcolor}
\usepackage{bm}
\usepackage{hyperref}
\usepackage{mathabx}
\usepackage{tikz-cd}
\usepackage{enumitem}
\usepackage{caption}
\usepackage{amssymb}
\setlist[enumerate]{label=\arabic*., ref=\arabic*}

\newtheorem{lemma}{Lemma}

\newtheorem{theorem}{Theorem}
\theoremstyle{definition}
\newtheorem{definition}{Definition}
\newtheorem{remark}{Remark}

\newtheorem{proposition}{Proposition}

\newcommand{\HH}{\mathbb{H}}

\newcommand{\G}{\mathcal{G}}
\newcommand{\T}{\mathbb{T}}
\newcommand{\E}{\mathcal{E}}
\newcommand{\p}{\partial}

\newcommand{\diam}{\mathrm{diam}}

\renewcommand{\Tilde}{\widetilde}

\title{On the minimal diameter of finite covers}

\author{Yifei Cai}
\address{Qiuzhen College, Tsinghua University, Beijing 100084, China}
\email{caiyf23@mails.tsinghua.edu.cn}

\author{Qiliang Luo}
\address{Qiuzhen College, Tsinghua University, Beijing 100084, China}
\email{lql23@mails.tsinghua.edu.cn}

\begin{document}
    
	\maketitle
	
	\begin{abstract} 
		We prove that every closed hyperbolic surface admits a sequence of finite covers with asymptotically optimal diameters.
	\end{abstract}
	
	\section{Introduction}
	
	Let $S$ denote a closed hyperbolic surface, and $\diam(S)$ the diameter of $S$.  An elementary area argument gives $\diam(S)\geq \log g(S)$, where $g$ denotes the genus.  
	\begin{definition}
	    A sequence of closed hyperbolic surfaces $S_k$ is said to be with asymptotically optimal diameter if
	    \[
	\diam(S_k)=\left(1+o(1)\right)\log g(S_k),\quad\text{as $k\to \infty$.}
	\]
	\end{definition}
	Budzinski-Curien-Petri \cite{2} showed that the minimal diameter of a closed hyperbolic surface of genus $g$ is asymptotically $\log g$ by constructing a sequence of surfaces with asymptotically optimal diameters. They do so by selecting surfaces from moduli spaces for all genus. We show a significant strengthening: one need only consider the set of covers of \emph{any} fixed hyperbolic surface.
	
	\begin{theorem}\label{main}
	Let $S$ be an arbitrary closed hyperbolic surface. There exists a sequence of finite covers $S_k$ of $S$ with asymptotically optimal diameter.
	\end{theorem}
	
	As an immediate consequence, we obtain:
	
	\begin{theorem}\label{main 0}
	    Given \emph{any} closed hyperbolic surface $S$, let $\mathrm{Comm}_S$ denote the set of surfaces commensurable to $S$. Then:
	    $$\inf_{F\in \mathrm{Comm}_S}\frac{\diam(F)}{\log g(F)}=1,$$
	\end{theorem}

	\subsection{Acknowledgments}
	The authors are extremely grateful to Yi Huang for many helpful suggestions on the earlier draft of this paper. The authors would also like to thank Yunhui Wu and Binbin Xu for telling us about \cite{2}. The second author would like to thank Vladimir Markovi\'c for his support and useful suggestions.

	\section{Constructing asymptotically minimal finite covers}
	\label{proof of the main}
	In this section, we prove Theorem \ref{main} by assuming the Random Cover Theorem, which is proved in Section \ref{probabilistic argument}. There are two key ideas behind the proof of Theorem \ref{main}: 
	\begin{enumerate}
	\item we use Kahn--Markovi\'c's theory \cite{4} to ensure that some finite cover of $S$ has a good pants decomposition structure (see Definition~\ref{good-pants-decom});
	\item we take the pants decomposition graph for such a good pants decomposition, and adapt ideas from the random cover theory of graphs \cite{1} to probabilistically\footnote{I.e.: non-explicitly.} construct finite covers of $S$ with minimal diameter growth rates.
		\end{enumerate}
	\subsection{Good pants decomposition}
	Let $S$ denote a hyperbolic surface. Let $\Gamma$ denote a pants decomposition of $S$, which is a maximal collection of disjoint simple closed curves. Then for every cuff $\gamma\in\Gamma$ we can define the reduced Fenchel-Nielsen coordinates $(\ell(\gamma),s(\gamma))$ from \cite{KM-surface group}, where $\ell(\gamma)$ is the length of the geodesic curve $\gamma$, and $s(\gamma)$ is the reduced twist parameter which lives in the interval $[0,\ell(\gamma)/2]$ (first reduce to $(-\ell(\gamma)/2,\ell(\gamma)/2]$ then take absolute value).
	
		\begin{definition}[Good pants decomposition]\label{good-pants-decom}
		Let $a>10$ be a given constant. Let $S$ be a hyperbolic surface with a pants decomposition $\Gamma$. The pants decomposition $\Gamma$ is said to be $a$-good if for every $\gamma\in\Gamma$ 
		\begin{itemize}
			\item the length of $\gamma$ satisfies $a\leq \ell(\gamma)\leq a+1$;
			\item the (reduced) twist parameter of $\gamma$ satisfies $2^{-1}\leq s(\gamma)\leq 2$. 
		\end{itemize}
	\end{definition}
	The following result, due to Kahn--Markovi\'c, ensures the existence of an $a$-good pants decomposition up to finite covers.
	\begin{proposition}[{\cite[Theorem 2.2]{4}}]\label{KM-prop}
	Let $S$ be any closed hyperbolic surface, then for every sufficiently large $a$, there exists a finite cover $S^a$ of $S$ which admits an $a$-good pants decomposition $\Gamma^a$. 	\end{proposition}
	
	\begin{definition}[Pants decomposition graph]
	    Let $\G$ be a trivalent graph defined as follows: 
	    
	    \begin{itemize}
	        \item The vertices of $\G$ correspond to the pairs of pants in the pants decomposition.
	        \item At each vertex, the three incident half-edges are marked by the boundary components of the corresponding pair of pants. Two half-edges are joined by an edge if and only if the corresponding boundary components are identified to form the same cuff of \(\Gamma\).
	    \end{itemize}
	   We call $\mathcal G$ a \emph{pants decomposition graph} of $\Gamma$ (It is called a “marked cubic graph” in \cite[Definition 3.5.2]{3}).
	\end{definition}

		\subsection{Random cover models for graphs and surfaces}
		
	    A random $m$-cover of a graph $\G$ is a probability distribution on the space of all degree $m$ covers of $\G$. We make use of a model for random covers of graphs taken from \cite[Definition~1]{1}, where such a cover is referred to as a \emph{random labeled $m$-cover}. 
	    \begin{definition}[Uniform edge model for graphs] For each edge $e$ in $\G$ we take an arbitrary orientation, so that $e=(u,v)$ for some vertices $u$ and $v$, and choose a random permutation $\sigma_e\in S_m$ uniformly and independently. Put $m$ vertices $u_1,\cdots,u_m$ above each vertex $u$ of $\G$, and put $m$ edges $e_i=(u_i,v_{\sigma_{e}(i)})$ above each oriented edge $e=(u,v)$. This provides a random (unoriented) graph $\G_m$ along with the canonical cover $\G_m\to\G$ sending $u_i$ to $u$ and sending $e_i$ to $e$.
	\end{definition}

    Associated to a graph cover, there is a natural surface cover. Let $S$ be a hyperbolic surface with a pants decomposition $\Gamma$, and let $\G$ be its pants decomposition graph. Let $q:\G'\to \G$ be a finite cover of graphs. The \emph{surface cover induced by $q$} is obtained by pulling back the pants decomposition data of $(S,\Gamma)$ along $q$: over each vertex of $\G'$, take a copy of the pair of pants corresponding to the image vertex in $\G$; for each edge of $\G'$, glue the corresponding boundary components by the same gluing map as the image edge in $\G$.

	\begin{definition}[Uniform edge model for surfaces]
	Let $\G$ be the pants decomposition graph of $\Gamma$. The random degree-$m$ cover $\mathcal G_m\to\mathcal G$ (in the uniform edge model) will induce a random cover for the surface $S$, denoted by $\mathcal S_m\to S$. 
	\end{definition}
	  
	As we will exclusively use the uniform edge model for this article, we henceforth refer to random covers of graphs/surfaces generated via these models simply as random covers. 
	
	\begin{theorem}[Random Cover Theorem]
	\label{main 2}
	Let $a$ be large enough and
		\[
		\delta_a:=1-5a^{-1}\log a.
		\]
		 Suppose $S$ is a hyperbolic surface equipped with an $a$-good pants decomposition $\Gamma$, whose pants decomposition graph is simple.  Then, for any $\epsilon>0$, with probability tending to $1$ as $m\to\infty$, a random cover $\mathcal S_m\to S$ of degree $m$ satisfies
		$$\diam(\mathcal S_m)\leq \left(\delta_a^{-1}+\epsilon\right)\log m.$$
	\end{theorem}
	
	In order to prove Theorem \ref{main 2}, we will first establish a uniform lower bound on volume growth for the corresponding pants tree in Section \ref{sec-vol}. Based on the volume growth bound and a probabilistic argument (known as the peeling process, see \cite{2,Peeling-BCP,Peeling-Budd}), we prove Theorem \ref{main 2}
	in Section \ref{probabilistic argument}. 
	\subsection{Proof of Theorem \ref{main}}
	Let $S$ be a closed hyperbolic surface. To prove the theorem, we construct a sequence of finite covers $S_k$ and estimate their diameters. 	
	
	\par For any integer $k$, let $a$ be a sufficiently large constant such that $0<\tfrac{1}{\delta_a}\leq 1+\tfrac{1}{2k}$, and $\epsilon=\tfrac{1}{2k}$. Let $S^a\to S$ be a finite cover from Proposition~\ref{KM-prop}. We shall use the elementary fact that every finite connected graph admits a finite connected cover which is simple (without loops and multi-edges). Applying this to the pants decomposition graph $\G^a$ of $\Gamma^a$, and pulling back the pants and gluing maps, we replace $(S^a,\Gamma^a)$ by a finite cover, which is still $a$-good by definition. According to Theorem \ref{main 2}, there exists a sufficiently large $m=m(k)\geq k$, and a degree $m$ finite cover $S^a_{m}\to S^a$ such that
		\[
		\diam(S^a_{m})\leq\left(\delta_a^{-1}+\epsilon\right)\log m\leq  \left(1+k^{-1}\right)\log m.
		\]
		Let $S_k=S^a_{m}$. By the Riemann-Hurwitz formula, $$ g(S^a_{m})=m\left(g(S^a)-1\right)+1\geq m.$$ This implies \[\diam(S_{k})\leq \left(1+k^{-1}\right)\log g(S_{k}).\] 
		We complete the proof of Theorem \ref{main} by combining this inequality with the area-based lower bound $ \diam(S_k)\geq \log g(S_k)$.
		\section{Volume growth of good pants trees}
	\label{sec-vol}

	 A \emph{pants tree} $T$ is an infinite-type hyperbolic surface equipped with a pants decomposition $\Gamma$, such that its pants decomposition graph $\T$ is isomorphic to the infinite trivalent tree. We say that $T$ is \emph{$a$-good}, if the pants decomposition $\Gamma$ is $a$-good.

     \begin{definition} Let $p_0\in T$ be an arbitrary chosen point. For $R>0$ let $\T(R)$ denote the subtree of $\T$, consisting of all vertices in $\T$ whose corresponding pair of pants has hyperbolic distance at most $R$ from $p_0$  in $T$. The \emph{volume growth order} of the pants tree $T$ is defined as
	\[
	\delta_T=\liminf_{R\to\infty}\frac{\log \big(\#\T(R)\big)}{R}.
	\]
	\end{definition}
	
	\begin{remark}
	 For convenience, we will conflate a tree with its vertex set, and a vertex in $\T$ with the corresponding pair of pants in $T$.
	 \end{remark}

	Standard rate comparison arguments show that $\delta_T$ is independent of the choice of the base point $p_0$. The goal of this section is to further show that $\delta_T$ has a lower bound which is independent of $T$. We will prove the following supermultiplicativity of the function $\#\T(R)$.
	\begin{proposition}
	\label{prop: supermulti}
	There exists a universal constant $C>0$ such that, for every $a$-good pants tree $T$, the following estimate
	\[
	\#\T\left(R+a/2\right)\geq Ca^{-2}e^{a/2}\#\T(R),
	\]
	holds for every $R>0$.
	\end{proposition}
	As a corollary, we have:
	\begin{theorem}
	\label{thm: volume-growth}
		There exists a universal constant $a_0>0$ such that, for every $a\geq a_0$ and every $a$-good pants tree $T$, its volume growth order satisfies
		\[
		\delta_T> \delta_a:=1-5a^{-1}\log a.
		\]
	As a consequence, there exists a constant $c_a>0$, such that for all sufficiently large $R>0$, we have $$\#\T(R)\geq c_ae^{\delta_aR}.$$
	\end{theorem}
    \begin{proof}
	 By taking logarithm to the estimate in Proposition \ref{prop: supermulti}, we obtain \[\log \#\T(R)\geq a/2-\log(a^2/C)+\log \#\T(R-a/2).\] 
    By repeating the above $\lfloor 2R/a\rfloor$ times and
    dividing by $R$, we get
    \[
    \frac{\log \#\T(R)}{R}\geq \frac {\lfloor 2R/a\rfloor}{R}\left(a/2-\log(a^2/C)\right).\]

    By letting $R$ tend to infinity, we get $\delta_T\geq 1-2a^{-1}\log (a^2/C)$. Letting $a_0=1+C^{-2}$ this completes the proof.
    \end{proof}
	\subsection{Geometry of good pants trees}
	 Let $T$ denote an $a$-good pants tree. For any two cuffs $\gamma_1,\gamma_2\in \Gamma$, let $d_{tree}(\gamma_1,\gamma_2)$ denote the combinatorial distance on the tree $\T$ given by counting the number of pairs of pants between $\gamma_1$ and $\gamma_2$. Let $d_{hyp}(\gamma_1,\gamma_2)$ denote the hyperbolic distance on $T$. The next lemma, essentially due to Kahn--Markovi\'c \cite{KM-surface group}, illustrates the geometry of good pants trees by comparing the two distances $d_{tree}$ and $d_{hyp}$.
	 \begin{lemma}[Metric Comparison Lemma]
	 \label{tree-hyp}
		There exists a universal constant $c>0$ such that the following holds. In an $a$-good pants tree $T$, any two cuffs $\gamma_1,\gamma_2$ satisfy:
		\[d_{tree}(\gamma_1,\gamma_2)\leq ca\cdot d_{hyp}(\gamma_1,\gamma_2).\]
	\end{lemma}
	\begin{proof}
	Let $\pi:\mathbb H^2\to T$ be the universal cover of the pants tree. Let $|\Gamma|$ be the union of all cuffs in $T$, and let $\lambda_\Gamma$ denote the geodesic lamination $\pi^{-1}(|\Gamma|)$ on $\mathbb H^2$. It suffices to prove that any geodesic arc $\sigma\subset \mathbb H^2$ passes through at most $\lfloor ca\ell(\sigma)\rfloor$ geodesic leaves in $\lambda_\Gamma$, for some uniform constant $c>0$.  
    
    We invoke:
	\begin{lemma}[{\cite[Lemma 2.2]{KM-surface group}}]\label{tech}
		
		Let $\sigma$ be a geodesic segment in $\mathbb{H}^2$ of length less than $e^{-5}$, which transversely intersects $\lambda_\Gamma$. Then there exists a universal constant $c_1>0$, such that $\#(\sigma\cap \lambda_\Gamma)\leq c_1a$.
	\end{lemma}
	Thus any length $e^{-5}$ geodesic arc in $\mathbb H^2$ passes through at most $c_1a$ many geodesic lines in the geodesic lamination $\lambda_\Gamma$. For general geodesic arc $\sigma$, by dividing it into length $e^{-5}$ subarcs,  we obtain that $\sigma$ passes through at most
	$\lfloor c_1ae^5\ell(\sigma)\rfloor$
	geodesic lines in $\lambda_\Gamma$. Letting $c=c_1e^5$, this completes the proof.
	\end{proof}
	
	\begin{remark}
	Lemma \ref{tree-hyp} can be viewed as a reformulation of Kahn--Markovi\'c's technical lemma (Lemma \ref{tech}), which was stated for the case when the cuff lengths of $\Gamma$ are exactly equal to $a$, and the twist parameters of $\Gamma$ are exactly equal to $1$. But their argument holds almost verbatim, when $\ell(\gamma )\in [a,a+1]$, and $s(\gamma)\in [2^{-1},2]$, $\forall \gamma \in \Gamma$.
	\end{remark}
	
	\begin{remark}
    The Metric Comparison Lemma crucially requires that all the twists are uniformly bounded away from zero, and fails when there are sufficiently many zero twists. Since the lengths of the seams for $a$-good pants behave like $e^{-a/4}$, we can at best expect $d_{tree}\leq e^{a/4}d_{hyp}$ when all twists are zero.
	\end{remark}

	\subsection{Boundary pairs of pants}
	In order to estimate the number of pairs of pants in $\T(R)$, we introduce two subsets $S^\prime(R)\subset S(R)\subset \T(R)$. Then we dominate $\# \T(R)$ by the number of pairs of pants in $S^\prime(R)$. This is established in Lemma \ref{lem: bad-good-est}.
    
    \begin{definition}
        A pair of pants $P$ in $\T(R)$ is said to be a \emph{boundary pair of pants} if 
	    \[
	    \deg_{\T(R)}(P)\leq 2.
	    \]
	   Let $S(R)$ denote the set of all boundary pairs of pants in $\T(R)$.
    \end{definition}

	\begin{lemma}
	\label{lem: S-estimate}
    For every $R>0$ we have
	$2 \#S(R)\geq \#\T(R).$
	\end{lemma}
	
	\begin{proof}
    Let $n_i$ be the number of vertices of degree $i$ in $\T(R)$. Since $\T$ is trivalent, $i$ can be 1,2 or 3.  Then $$\#\T(R)=n_1+n_2+n_3.$$ 
    Since $\T(R)$ is a tree, the handshaking lemma gives
    $$2\#\T(R)=n_1+2n_2+3n_3.$$ Combining the above two equalities yields $n_1=n_3$, and hence $$2\#S(R)=2(n_1+n_2)=n_1+2n_2+n_3\geq  n_1+n_2+n_3=\#\T(R).$$
	\end{proof}
	
	Let $P_0$ be the pair of pants in $\T(R)$ containing $p_0$. The tree $\T(R)$ is rooted at $P_0$ from now on. For every non-root pair of pants $P\in\T(R)$ we denote by $\gamma_P$ the cuff of $P$ that is shared with the parent of $P$, called the \emph{incoming cuff} of $P$.
	 
	 \begin{definition}
	     A boundary pair of pants $P\in S(R)$ is said to be \emph{useful}, if either 
	 \begin{enumerate}
	     \item $\gamma_P$ is fully contained in the hyperbolic ball centered at $p_0$ with radius $R$,
	    \item or $\deg_{\T(R)}(P)=1$, which is said to be a leaf pants.
	 \end{enumerate}
	   Let $S^\prime(R)\subset S(R)$ denote the set of all useful boundary pairs of pants. 
	 \end{definition}

	To dominate $\#\T(R)$ by the number of useful pairs of pants, we prove the following lemma using the Metric Comparison Lemma.  
     
     \begin{lemma}\label{Sdistr}
     Assume $P\in S(R)\setminus S^\prime(R)$. Then for any descendant pair of pants $L\in \T(R)$ of $P$ we have 
     \[d_{tree}(\gamma_P,\gamma_L)\leq Ca^2,\] for some universal $C>0$.
     \end{lemma}
     
     \begin{proof}
     	 We prove the lemma by bounding $d_{hyp}(\gamma_P,\gamma_L)$ from above. Let $\Tilde p_0\in\HH$ denote a lift of the base point $p_0$, and $B_{\Tilde p_0}(R)$ the radius $R$ hyperbolic ball. Let $\Tilde\gamma_L$ be a lift of the incoming cuff $\gamma_L$ meeting $B_{\Tilde p_0}(R)$, and let $\Tilde\gamma_P$ be a lift of $\gamma_P$ separating $\Tilde\gamma_L$ from the center point $\Tilde p_0$. 
     	 \begin{figure}[h]
	    \centering
	    \includegraphics[width=0.8\textwidth]{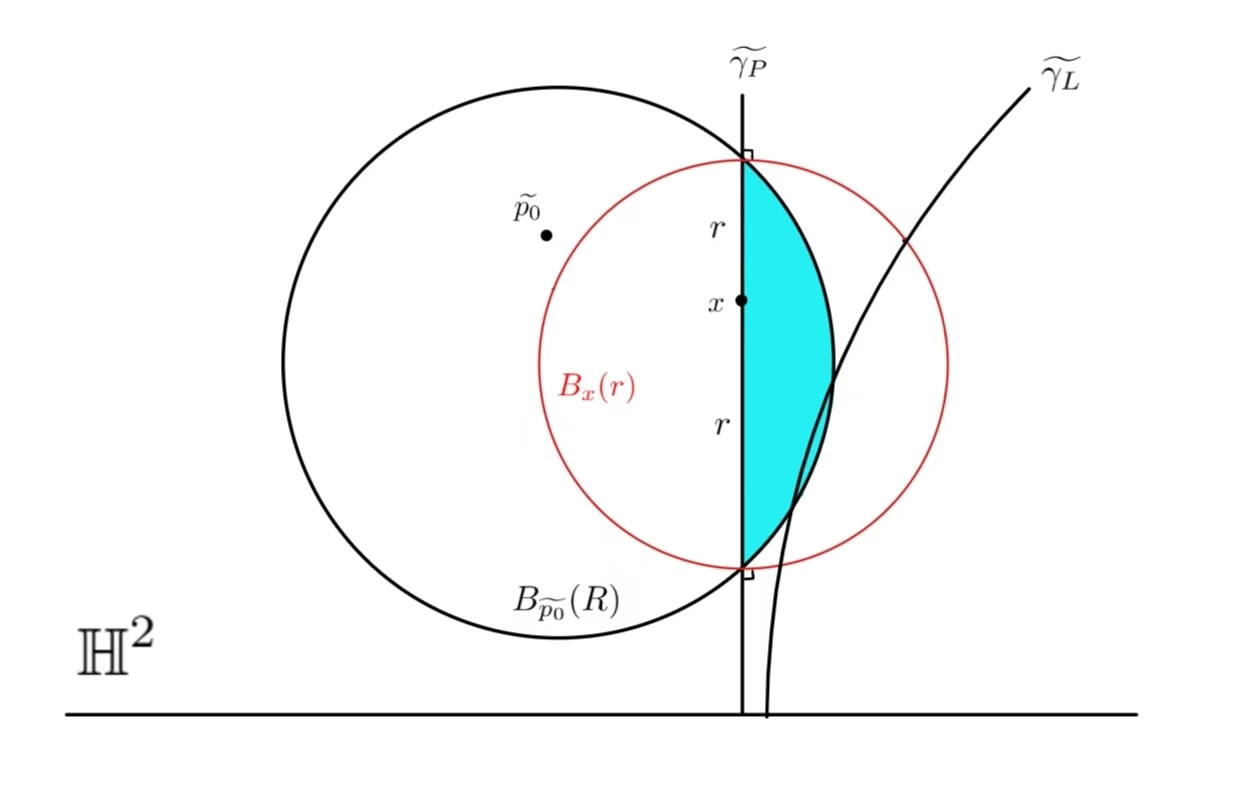}
	    \caption{A lift to the universal cover.}
	    \label{fig1}
	\end{figure}

Since $P\notin S^\prime(R)$, by definition $\gamma_P$ is not fully contained in the radius $R$ hyperbolic ball in $T$, and hence we have the following estimate
\begin{equation}
\label{arcleng}
    \ell\left(B_{\Tilde{p_0}}(R)\cap \Tilde {\gamma_P}\right)\leq \ell(\gamma_P)\leq  a+1.
\end{equation} 
	This implies that (see Figure \ref{fig1}, where the blue part is included in the red ball $B(x,r)$)
	\[
	d_{\mathbb H}(\Tilde\gamma_P,\Tilde\gamma_L)\leq \frac{\ell\left(B_{\Tilde{p_0}}(R)\cap \Tilde {\gamma_P}\right)}{2}\leq a.
	\]
		Therefore, we have $d_{hyp}(\gamma_P,\gamma_L)\leq a$. By applying the Metric Comparison Lemma we prove the lemma. 
     \end{proof}
     
     We are going to dominate $\#\T(R)$ by the number of useful boundary pants.     
    \begin{lemma}
    \label{lem: bad-good-est}
	There exists a universal constant $C>0$, such that for any $a$-good pants tree $T$, the following holds
		$$\#S'(R)\geq Ca^{-2}\#\T(R).$$
	\end{lemma}
	\begin{proof}
		
		 Let 
		$\mathrm{Leaf}:S(R)\setminus S^\prime(R) \to S^\prime(R)$
		be a map that sends each $P$ to a descendent leaf pants of $P$. Let $C_1$ be the constant in Lemma \ref{Sdistr}.
		We claim that each leaf pants $L$ has at most $C_1a^2$ pre-images under this map. In fact, if $\mathrm{Leaf}(P)=L$, then $P$ is an ancestor of $L$ and there are at most $C_1a^2$ edges between $P$ and $L$. Along the unique descendant path from the root $P_0$ to $L$, there are at most $C_1a^2$ such vertices. This proves the claim. Since each fiber of the map $\mathrm{Leaf}$ has at most $C_1a^2$ pairs of pants, we get 
		\[ C_1a^2\#S'(R)\geq \#S(R)-\#S'(R).\]
		 This further implies that
		\[
		S^\prime(R)\geq \frac{1}{1+C_1a^2}\#S(R)\geq \frac{a^2}{2(1+C_1a^2)}a^{-2}\#\T(R),
		\]
		where the second inequality comes from Lemma  \ref{lem: S-estimate}. Letting $C=\tfrac{1}{2(1+C_1)}$ completes the proof.
	\end{proof}
	Note that Lemma \ref{Sdistr} and Lemma \ref{lem: bad-good-est} use the Metric Comparison Lemma, hence only apply to the pants trees with twists uniformly bounded away from zero.
	 
	\subsection{Proof of Proposition \ref{prop: supermulti}}

	Proposition \ref{prop: supermulti} follows from Lemma \ref{lem: bad-good-est} and the following inequality:
    there exists a universal constant $C>0$ such that for every $R>0$,
	\begin{equation}\label{a/2ite}
	    \#\T(R+a/2)\geq Ce^{a/2}\#S^\prime(R).
	\end{equation}
 
We are going to prove (\ref{a/2ite}) by using a classical area-estimate argument: Let $\Omega\subset T$ denote the radius $R+a/2$ hyperbolic ball centered at $p_0$ in $T$. The area of a single hyperbolic pair of pants is equal to $2\pi$. Since $\Omega$ is included in the union of all pairs of pants from $\T(R+a/2)$, we have 
        \begin{equation}\label{vol-E1}
            2\pi \#\T\left(R+a/2\right)\geq \mathrm{Area}\left(\Omega\right),
        \end{equation}
        For any useful boundary pair of pants $P\in S^\prime(R)$ we construct a subdomain $\Omega(P)\subset\Omega$ as follows.
        \begin{figure}[ht]
        \centering
       \includegraphics[width=0.8\textwidth]{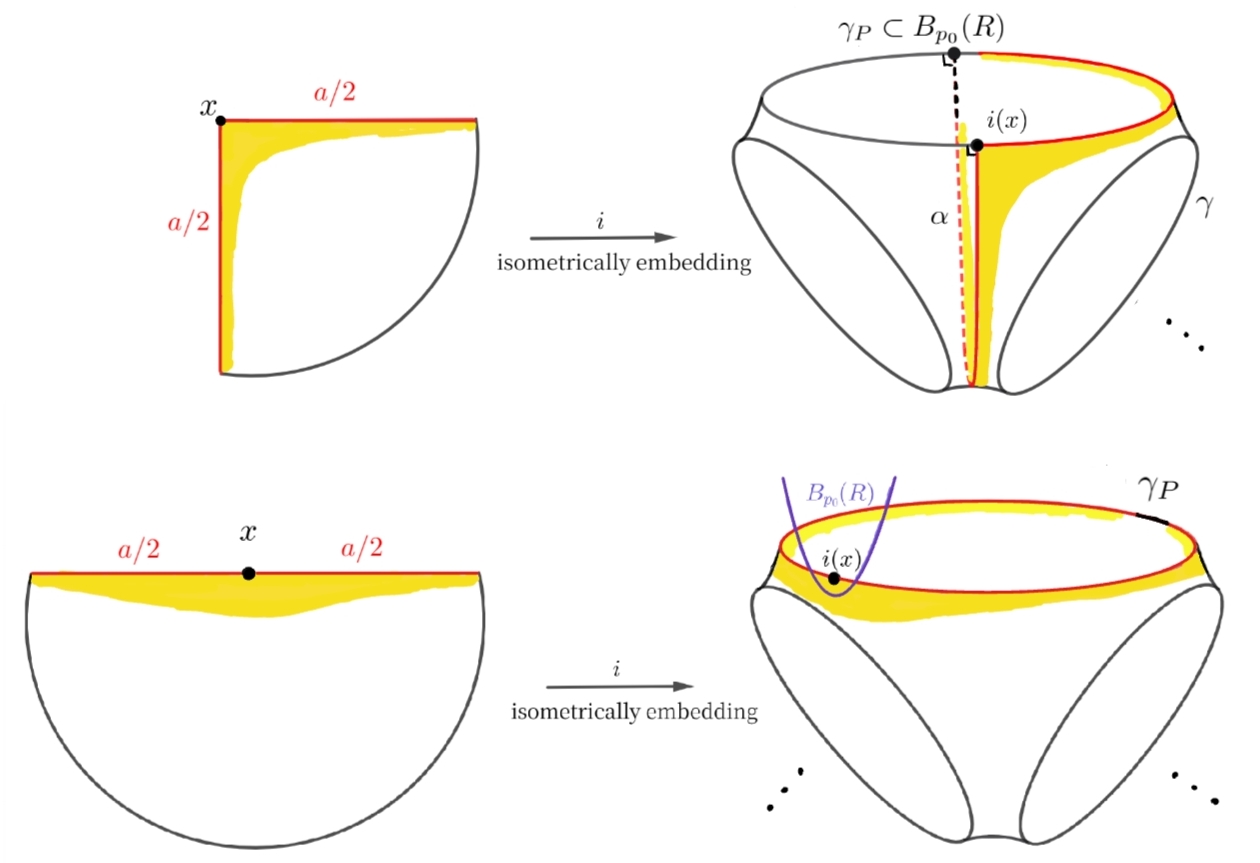}
        \captionsetup{font={small}}
        \caption{ The map $i$ is an isometrically immersing from radius-$a/2$ half ball (or $1/4$ ball) to the pants tree, and is determined by mapping $x$ to $i(x)$ and the yellow part to  the yellow part. That the map $i$ is embedding follows from the fact \cite[Proposition 2.28]{M} that the injectivity radius of $T$ is at least $a/2$ .  
        }
        \label{fig2}
        \end{figure}
      \begin{enumerate}
          \item  If $\gamma_P$ is fully contained in the radius-$R$ ball of $T$. Then $\Omega(P)$ is the image of  the upper isometric embedding $i$ from Figure \ref{fig2}, where $\gamma$ is the cuff of $P$ avoiding the radius-$R$ ball, and $\alpha$ is the orthogonal geodesic from $\gamma_P$ to itself. 
          \item If $P$ is a leaf pants. Then $\Omega(P)$ is the image of the lower isometric embedding $i$ from Figure \ref{fig2}, where $x\in\gamma_P$ is a point lying in the radius-$R$ ball in $T$. 
      \end{enumerate}
      Since $\Omega(P)$ leads outside $\T(R)$, those subregions are mutually disjoint. Combining this with the area estimate of hyperbolic ball $\mathrm{Area}(\Omega(P))=\frac{\pi}{2}(\cosh \frac a2-1)\geq ce^{a/2}$, for some universal $c>0$, we have
        \begin{equation}\label{vol-E2}
            \mathrm{Area}\left(\Omega\right)\geq \sum_{P\in S'(R)}\mathrm{Area}\left(\Omega(P)\right)\geq ce^{a/2}\#S^\prime(R).
        \end{equation}
        Combining  (\ref{vol-E1}) and (\ref{vol-E2}) implies (\ref{a/2ite}). This completes the proof.

\section{Diameter of random cover}
\label{probabilistic argument}
	In this section, we prove the Random Cover Theorem. Let $S$ be a closed hyperbolic surface with an $a$-good pants decomposition $\Gamma$, whose pants decomposition graph $\G$ is simple. We want to show that for any $\epsilon>0$ the degree $m$ random cover $\mathcal S_m$ has diameter $$\diam(\mathcal S_m)\leq \left(\frac{1}{\delta_a}+\epsilon\right)\log m,$$
	with high probability as $m\to\infty$.
	\subsection{Pants exploration with algorithm $\mathcal A$} 
	Let $\G$ denote the pants decomposition graph of $(S,\Gamma)$. Putting $m$ pairs of pants above each vertex of $\G$, let $\mathcal P_m$ denote the collection of these pairs of pants. The pairs of pants in $\mathcal{P}_m$ above the vertex $v\in \G$ is said to be \emph{of type $v$}. Recall that the degree $m$ surface covers of $S$ are constructed by assembling pants in $\mathcal P_m$ along edges of $\G$. We generalize this notion by allowing a subset of $\mathcal P_m$ engaged in the assembling.
	\begin{definition}
		An \emph{assembly surface} $E$ is a surface constructed by assembling pairs of pants from a subset of $\mathcal P_m$ along edges of $\G$. An assembly surface is equipped with the induced pants decomposition structure, whose pants are labeled by $\mathcal P_m$. 
	\end{definition}
	
	\begin{remark}
		Since an assembly surface  $E$ is constructed by assembling along edges, there is a natural map $E\to S$. Surface covers form a special class of assembly surfaces. 
	\end{remark}

	 An assembly surface $E$ may have boundary cuffs, and we let $\partial E$ denote the set of boundary cuffs. For assembly surfaces $E$ and $E^\prime$ if there exists a label-preserving embedding $E\hookrightarrow E^\prime$, then $E$ is said to be an assembly subsurface of $E^\prime$, denoted by $E\subset E^\prime$. If in addition $E$ is a single pair of pants, then $E$ is said to be a pair of pants in $E^\prime$, denoted by $E\in E^\prime$. 
	\subsubsection{Pants exploration}
	    Let $S_m$ be a degree $m$ surface cover of $S$, and let $E\subset S_m$ be an assembly subsurface with nonempty boundary.  Given a boundary cuff $\gamma\in \partial E$ we are going to construct a new assembly surface $E^\prime$ such that $
	 E\subset E^\prime\subset S_m$.  To do so, let $P$ be the pair of pants in $S_m$ that has the orientation inversion of $\gamma$ as a boundary. Two events may happen:
	\begin{enumerate}
	    \item Good step: $P\notin E$. In this case $E^\prime$ is constructed by gluing $P$ to $E$ along the boundary $\gamma$. 
	    \item Bad step: $P\in E$. In this case both of $\gamma$ and its orientation inversion are in the boundary $\partial E$. We construct $E^\prime$ by gluing $\gamma$ and its orientation inversion.
	\end{enumerate}
	The new assembly surface $E^\prime$ is said to be \emph{obtained from $E$ by exploring beyond the boundary cuff $\gamma$}. By iterating this operation we can explore the surface $S_m$. 
	\subsubsection{Exploration with algorithm}
	\begin{definition}
		An algorithm $\mathcal A$ is a map that, for any assembly surface $E$ with nonempty boundary and any pair of pants $P\in E$, selects a boundary cuff $\mathcal A(E,P)\in\partial E$.
	\end{definition}
 	\begin{definition}
 	    For $P\in\mathcal P_m$ and a surface cover $S_m$, the \emph{pants exploration  of $S_m$ with algorithm $\mathcal A$ rooted at $P$} is a sequence of assembly surfaces
 	\[
 	P=E_0\subset E_1\subset E_2\subset\cdots\subset S_m,
 	\]
 	where the assembly surface $E_{i+1}$ is obtained from $E_i$ by exploring beyond the boundary cuff $\mathcal A(E_i,P)\in\partial E_i$. 
 	\end{definition}

 	\par  We denote by $\E_i$ the pants exploration of the random surface $\mathcal S_m$ rooted at some $P\in\mathcal P_m$. Hence $\E_i$  is a random variable. The following proposition establishes the probability of getting a bad $i$-th step from $\E_{i-1}$ to $\E_i$.
	 \begin{proposition}\label{badprob}
	 	There exists a universal constant $c>0$ such that for $i\leq m/4$ the probability that the step i is a bad step satisfies
	 	\[
	 	\mathbb P(\text{step i is bad})\leq c\frac{i}{m}.
	 	\]
	 \end{proposition}
	\begin{proof}
	According to the formula of total probability
	\[
	\mathbb P(\text{step i is bad})=\sum_{E}\mathbb P\left(\text{step i is bad}|\, \E_{i-1}=E\right)\mathbb P\left(\E_{i-1}=E\right),
	\]
	where the summation is taken over the set of all assembly surfaces. To prove the proposition, it suffices to prove that the probability of getting a bad step i given $\E_{i-1}=E$ is bounded from above by $c\tfrac{i}{m}$, for some universal $c>0$. 
	\par  Given $E$, let $e\subset\G$ be the edge type of the cuff $\mathcal A(E,P)$. Let $v$ denote the type of the pair of pants obtained by exploring beyond $\mathcal A(E,P)$. Let $E|_v$ denote the set of all pairs of pants in $E$ of type $v$, and let $\partial E|_v$ denote the set of elements in $E|_v$ whose type $e$ cuff lies in $\partial E$. Thus
	\[
	\mathcal P_m|_v\supset E|_v\supset \partial E|_v,
	\]
	where $\mathcal P_m|_v$ denotes the set of all type $v$ pairs of pants.  Under the condition $\E_{i-1}=E$, at step i we explore a pair of pants among $\mathcal P_m|_v-(E|_v-\partial E|_v)$. The event 
	\[(\text{step i is bad})\cap( \E_{i-1}=E)\]happens if and only if the pair of pants explored at step i is among $\partial E|_v$. According to the uniform edge model, we have
	\[
	\mathbb P\left(\text{step i is bad}|\, \E_{i-1}=E\right)=\frac{\#\partial E|_v}{m-\# E|_v+\#\partial E|_v}.
	\]
	Since each step explores at most one pair of pants, we have
	\[
	\#\partial E|_v\leq \# E|_v\leq \#\E_{i-1}\leq i.
	\]
	This implies
	\[
	\mathbb P\left(\text{step i is bad}|\, \E_{i-1}=E\right)\leq \frac{i}{m-i}\leq \frac{4}{3}\frac{i}{m},
	\]
	where the last inequality holds provided $i\leq m/4$. Letting $c=4/3$, the probability of getting a bad step i given $\E_{i-1}=E$ is therefore bounded from above by $c\tfrac{i}{m}$. Combining this with the formula of total probability, we complete the proof.
	\end{proof}
	\subsection{Nearest cuff algorithm}
	In this subsection, we discuss the deterministic properties of assembly surfaces obtained by exploring with a nearest cuff algorithm. 
	 
	\begin{definition}
		A nearest cuff algorithm $\mathcal A$ is an  algorithm such that, for any assembly surface $E$ and any $P\in E$, the  boundary cuff $\mathcal A(E,P)$ minimizes the hyperbolic distance to $P$ in $E$ among all boundary cuffs of $E$.
	\end{definition}
	 
	Let $d_E$ denote the hyperbolic metric on the assembly surface $E$.
	 Assume $E^\prime\subset E$. Let $R^+(E^\prime,E)$ denote the Hausdorff distance between $E^\prime$ and $E$ with respect to $d_E$. Hence, if $R^+(E^\prime,E)\leq R$ then $E$ is exhausted by the $R$-neighborhood of $E^\prime$. Let $R^-(E^\prime,E)$ denote the distance from $E^\prime$ to the boundary of $E$.

	\begin{lemma}
	\label{lem: nearest}
		Let $\Delta_a$ be the maximal diameter of $a$-good pairs of pants. For any pants exploration with nearest cuff algorithm
		\[
		E_0\subset E_1\subset\cdots E_i \subset \cdots \subset E_j
		\]
		we have
		\[
		R^-(E_i, E_j)\leq R^+(E_i,E_j)\leq R^-(E_i, E_j)+2\Delta_a.
		\]
	\end{lemma}
	\begin{proof}
		The first inequality follows from the definition of $R^+$ and $R^-$. We prove the second inequality by contradiction. Assume that $R^+>R^-+2\Delta_a$ (here we omit $(E_i,E_j)$ to simplify notation).  This implies that there exist two consecutive pairs of pants $P,P^\prime\in E_j$ such that $d_{E_j}(E_i,P),d_{E_j}(E_i,P^\prime)>R^-$. This can't happen in  exploration with nearest cuff algorithm, a contradiction.
	\end{proof}
	For an edge $e\in E(\G)$ and an assembly surface $E$, let $I_e(E)$ denote the number of interior cuff of $E$ of type $e$. A priori, an assembly surface $E$ could concentrate over only part of $S$, such that $I_e(E)$ is much less than the cuffs in $E$. The next lemma shows that this cannot happen if $E$ is obtained by pants exploration with nearest cuff algorithm: every edge type occurs with positive density among the interior cuffs of $E$.
	\begin{lemma}[Equidistribution Lemma]
	\label{lem: equidistribution}
	There exists a constant $c>0$ depending only on $S$ and $\Gamma$ such that the following holds. Let $e$ denote an edge of $\G$. For any assembly surface $E$  obtained by exploring with nearest cuff algorithm, the number of interior cuffs in $E$ of type $e$ is bounded from below   
	\[ I_e(E)\geq c\#E-\frac 1c.\]
	
	\end{lemma}
	
	\begin{proof}
		Take a spanning subtree $\mathcal H$ of $\G$ containing $e$, then the cuff corresponding to $e$ is interior to the associated subsurface $H$. Let $D$ be the diameter of the surface $H$. Choose an auxiliary surface cover $S_m$ such that $E\subset S_m$. For each pair of pants $P\in S_m$, let $H_P\subset S_m$ be the unique lift of $H$ containing $P$.  Distinct pants of the same type lie in distinct copies of $H$. Every copy of $H$ contains exactly one pair of pants of every type and one interior cuff of type $e$.
		
		\par Assume $E^\prime$ is an assembly subsurface of $E$ such that $R^-(E^\prime,E)> D$. Then for any pair of pants $P\in E^\prime$ the lift $H_P$ is included in $E$. This implies that
		\begin{equation}
		\label{E100}
			I_e(E)\geq\text{the number of lifts included in $E$}\geq\# E^\prime|_{v},
		\end{equation}
		for every vertex $v\in\G$.   
		\par Since $E$ is obtained by exploring, there exists a sequence of assembly surfaces $E_0\subset E_1\subset\cdots\subset E$. Let $E_i$ be the first assembly surface in this sequence such that $ R^-(E_{i},E)\leq D$. According to Lemma \ref{lem: nearest}\begin{equation}
		\begin{aligned}
			R^+(E_i,E)\leq R^-(E_{i},E)+2\Delta_a\leq D+2\Delta_a.
		\end{aligned}
	\end{equation}
	We claim that the above estimate implies
	\begin{equation}\label{E1}
	    \# E_i\geq \frac 1N\# E,
	\end{equation}
	for some $N>0$ depending only on $S$ and $\Gamma$.  In fact, the surface $E$ is covered by the $(D+2\Delta_a)$-neighborhood of $E_i$, hence covered by the union of $(D+2\Delta_a)$-neighborhoods of pairs of pants in $E_i$. The constant $N$ is the uniform upper bound of the number of pairs of pants in these neighborhoods. This proves the claim.
	\par Now we are going to complete the proof. If $i=0$, then $N \geq \#E$. If $i\geq 1$, then $R^-(E_{i-1},E)>D$ according to the definition of $E_i$. By replacing $E^\prime$ with $E_{i-1}$ in (\ref{E100}) and  taking sum over $v\in\G$, we obtain
		\begin{equation}
		\label{E101}
		I_e(E)\geq \frac{1}{k}\#E_{i-1},
		\end{equation}
		where $k=\#\G$. Combining (\ref{E1}) and (\ref{E101}) implies 
		\[I_e(E)\geq\frac 1k\left(\#E_i-1\right)\geq \frac{1}{Nk}\# E-\frac 1k.\] 
		Letting $c=(kN)^{-1}$, this completes the proof.
	\end{proof}
	\subsection{Efficient pants explorations}
	Let $S_m$ be a degree $m$ cover of $S$. From now on, we are mostly interested in the pants exploration of $S_m$ with nearest cuff algorithm
	\[
 	P=E_0\subset E_1\subset E_2\subset\cdots\subset E_{ter},
 	\] 
 	where $\lfloor m^{1/2}\log m\rfloor$ pairs of pants are found in the terminal exploration subsurface $E_{ter}\subset S_m$. If $S_m$ is connected, then the terminal exploration subsurface $E_{ter}$ can always be found before $3\lfloor m^{1/2}\log m\rfloor$ steps. If not, then the terminal exploration subsurface may not exit.
	\begin{definition}
	    Let $\epsilon>0$. A pants exploration of $S_m$
	    \[
	    P=E_0\subset E_1\subset\cdots\subset E_{ter}
	    \]is said to be \emph{$\epsilon$-efficient}, if it 
	    satisfies
	    \begin{enumerate}
	    \item First phase condition: Let $M=\lfloor m^{1/2-\epsilon}\rfloor$, $k_1=\lfloor3/\epsilon\rfloor$. The sequence of explorations 
	    \[
	    E_0\subset E_1\subset\cdots\subset E_{M}
	    \]
	    gets at most $k_1$ bad steps;
	    \item Second phase condition: Let $k_2=\lfloor\log^3m\rfloor$. The sequence of explorations
	    \[
	    \ E_{M}\subset\cdots\subset E_{ter}
	    \]
	    gets at most $k_2$ bad steps.
	\end{enumerate}
	 
	\end{definition}

	As before, we write $\E_{ter}$ to be the terminal exploration subsurface in the random surface $\mathcal{S}_m$ rooted at some $P\in\mathcal P_m$.	The constants $M,k_1,k_2$ in the above definition are carefully chosen such that the following two lemmas hold.
	\begin{lemma}\label{eff}
			Assume the pants decomposition graph $\G$ is simple. Then with probability tending to $1$ as $m\to\infty$, a random $m$-cover $\mathcal S_m$ of $S$ satisfies the following properties:
	\begin{enumerate}
		\item The surface $\mathcal S_m$ is connected;
		\item For every $P\in \mathcal P_m$, the pants exploration with nearest cuff algorithm rooted at $P$ is $\epsilon$-efficient;
	\end{enumerate}
	\end{lemma}
	\begin{proof}Condition 1 follows from the next proposition. 
	\begin{proposition}[{\cite[Theorem 1]{1}}]\label{conn}
	Let $\G$ be a connected simple graph with minimal degree $d\geq 3$. Then with probability $1-o(1)$ as $m\to \infty$, a random $m$-cover of $\G$ is connected. \qed
	\end{proposition}
    Next we compute the probability of $\epsilon$-efficiency. By Proposition \ref{badprob},
	$$\mathbb P(\text{step i is bad})\leq c\frac{i}{m}.$$
		A union of all $k$-tuples of steps shows that the probability of making at least $k$ bad steps among the first $N$ steps is bounded above by  $$\binom{N}{k}\times \left(c\frac{N}{m}\right)^k\leq \frac{1}{k!}\left(c\frac{N^2}{m}\right)^k.$$
	   So the probability of getting $k_1$ bad steps in the first $M$ steps is bounded above by $$\frac{1}{k_1!}\left(c\frac{m^{1-2\epsilon}}{m}\right)^{k_1}=O(m^{-2\epsilon k_1})=o(m^{-3}).$$
	   Similarly, the probability of getting $k_2$ bad steps in the second phase condition is bounded above by $$\frac{1}{(\log^3m)!}\left(9c\frac{m\log^2m}{m}\right)^{\log^3m}\leq \left(\frac{c'}{\log m}\right)^{\log^3m}=o(m^{-3}). $$
	    where the first inequality is by Stirling's formula. 
	    
	    From the calculation above we know that for any $P\in \mathcal P_m$, the pants exploration with nearest cuff algorithm rooted at $P$ is $\epsilon$-efficient with probability $1-o(m^{-3})$. Finally, applying a union bound over $\#\mathcal{P}_m=O(m)$ many $P$, one proves Condition 2.
	    \end{proof}
	    
	  The following lemma bounds the diameter of an efficient $\E_{ter}$.
	\begin{lemma}\label{expldist}
	 For any $\epsilon$-efficient pants exploration of $S_m$ with nearest cuff algorithm rooted at $P_0$ $$P_0=E_0\subset \cdots \subset E_{ter}\subset S_m,$$
	 the following inequality holds for all sufficiently large $m$
	 \[R^+\left(P_0,E_{ter}\right)\leq \frac{1}{2}\left(\frac{1}{\delta_a}+10\epsilon\right)\log m.\]
	\end{lemma}
	
	\begin{proof}
	Let $R_1=R^+(E_0,E_M)$, and $R_2=R^+(E_M,E_{ter})$.  We are going to estimate $R_1$ and $R_2$ separately. 
	\par According to the first phase condition, there must be at least one pair of pants $P\in E_M$ at graph distance at most $k_1$ from $P_0$, none of whose descendants in $E_M$ are explored in a bad step. The descendants of $P$ in $E_M$ forms a connected component $X$ of $E_M-\gamma_P$, where $\gamma_P$ is the incoming cuff of $P$. By doubling $X$ along the incoming cuff $\gamma_P$, we get a surface $\Tilde{X}$ which is an $a$-good finite pants tree. According to Lemma \ref{lem: nearest}, the finite pants tree $\Tilde X$ contains a $(R_1-C_1)$-ball for some $C_1>0$ depending only on $k_1$ and $a$. Theorem \ref{thm: volume-growth} then yields
	\[
	\#E_M\geq \#X= \frac{1}{2}\#\Tilde{X}\geq  e^{\delta_a(R_1-C_1)},
	\] 
	provided $m$ sufficiently large. Since $\#E_M\leq m^{1/2}$, the above estimate implies that $\tfrac{1}{2\delta_a}\log m+C_1\geq R_1$.
	\par We next estimate $R_2$ using the same argument. There are at least $M-3k_1$ boundary cuffs in $\partial E_M$. According to the second phase condition, there are at least $M-3k_1-2k_2$ boundary cuffs $\gamma$ among them, none of whose descendants in $E_{ter}$ are explored in a bad step. For any such $\gamma$ we denote by $X_\gamma$ the descendant connected component. By using the same doubling trick (and Theorem \ref{thm: volume-growth}) for each $X_\gamma$, we obtain the following estimate
	 \[
	 \#E_{ter}\geq \sum_{\gamma}\#X_\gamma\geq (M-3k_1-2k_2)e^{\delta_a(R_2-C_2)}\geq m^{1/2-3\epsilon}e^{\delta_a(R_2-C_2)},
	 \]
	for every sufficiently large $m>0$ and some $C_2>0$ depending only on $a$. Since $\#E_{ter}\leq m^{1/2}\log m$, we obtain $
	4\epsilon\log m+C_2\geq R_2.$ Since $R^+(E_0,E_{ter})\leq R_1+R_2$, combining the estimates on $R_1,R_2$ yields \[
	R^+(E_0,E_{ter})\leq \frac 12\left(\frac{1}{\delta_a}+8\epsilon\right)\log m+C_1+C_2.
	\]
	This completes the proof, because $m$ is sufficiently large.
	\end{proof}

	\subsection{Proof of Theorem \ref{main 2}}
	We are in a position to prove Theorem \ref{main 2}. Let $S_m$ be a surface cover. For any two points $p,p^\prime\in S_m$ we have
	\[
	d_{hyp}(p,p^\prime)\leq d_{hyp}(P,P^\prime)+2\Delta_a,
	\]
	where $d_{hyp}$ is the hyperbolic metric on $S_m$, and $P$ and $P^\prime$ are pairs of pants that $p$ and $p^\prime$ lie in, respectively. Assume $S_m$ satisfies the following properties:
	\begin{enumerate}
	    \item the surface $S_m$ is connected;
	    \item for any $P\in\mathcal P_m$ the terminal exploration of $S_m$ rooted at $P$ is $\epsilon$-efficient;
	    \item for any $P,P^\prime\in\mathcal P_m$ the terminal exploration $E_{ter}$ rooted at $P$ and $E^\prime_{ter}$ rooted at $P^\prime$ have nonempty intersection.
	\end{enumerate}
	Then according to Lemma \ref{expldist}, we have
	\[
	d_{hyp}(P,P^\prime)\leq R^+(P,E_{ter})+R^+(P^\prime,E_{ter}^\prime)\leq \left(\frac{1}{\delta_a}+10\epsilon\right)\log m,
	\]
	where the first inequality comes from the property that $E_{ter}$ and $E^\prime_{ter}$ have nonempty intersection. This implies that
	\[
	\diam(S_m)\leq \sup_{p,p^\prime}d_{hyp}(p,p^\prime)\leq \left(\frac{1}{\delta_a}+11\epsilon\right)\log m,
	\]
	provided $m$ sufficiently large. We \emph{claim} that Properties 1-3 hold, and hence the diameter estimate holds, for the random surface $\mathcal S_m$ with probability tending to $1$ as $m\to\infty$. This completes the proof.

	\begin{proof}[Proof of the claim]
	 We have shown in Lemma \ref{eff} that Properties 1 and 2 hold with high probability. It suffices to prove that the probability that Property 3 fails is negligible. Given $P,P^\prime\in\mathcal P_m$, let $\E$ denote the exploration rooted at $P$ and  $\E^\prime$ denote the exploration rooted at $P^\prime$. We are going to prove that
	 \[\mathbb P(\E_{ter}\cap\E^\prime_{ter}=\emptyset)=o(m^{-3}).\]
	 Therefore, the probability that Property 3 fails is $o(m^{-1})$.
	 \par Under the condition $\E_{i-1}\cap\E^\prime_{ter}=\emptyset$, the event $\E_i\cap\E^\prime_{ter}=\emptyset$ happens only if the explored pair of pants avoids $\partial\E^\prime_{ter}$. Assume we are exploring through a cuff of type $e$. Then the probability of avoiding $\partial\E^\prime_{ter}$ is bounded above by
	 \[
	 1-\frac{\#\partial\E^\prime_{ter}|_e}{m}.
	 \]
	 Since $\#\p\E_{ter}^\prime=(1-o(1))m^{1/2}\log m$, a pigeonhole argument shows that there exists an edge type $e$ such that
	 \[\#\p\E_{ter}^\prime|_e\geq c_1m^{1/2}\log m,\]
	 for some constant $c_1>0$. Therefore, at each step of $\E$ exploring through a cuff of type $e$, the probability of avoiding $\partial \E_{ter}^\prime$ is bounded above by 
	 \[ 1-\frac{\#\p\E_{ter}'|_e}{m}\leq 1-\frac{c_1m^{1/2}\log m}{m}.
	 \]
	 \par On the other hand, by Lemma \ref{lem: equidistribution},  the number of interior cuffs in $\E_{ter}$ of type $e$ is bounded from below by $c_2 m^{1/2}\log m,$ for some $c_2>0$. There must be at least $c_2 m^{1/2}\log m$ steps through a type $e$ cuff before reaching $\E_{ter}$. Hence, the probability that $\E_{ter}$ avoids $\E_{ter}^\prime$ is bounded above by: 
	 \[\left(1-\frac{c_1m^{1/2}\log m}{m}\right)^{c_2m^{1/2}\log m}\leq e^{-c_1c_2\log^2m}=o(m^{-3}).\] 
	\end{proof}


\begin{thebibliography}{}
	\bibitem{1}A. Amit and N. Linial, Random graph covers. I. General theory and graph connectivity, Combinatorica {\bf 22} (2002), no.~1, 1--18; MR1883559
	
	\bibitem{2}T. Budzinski, N. Curien and B. Petri, On the minimal diameter of closed hyperbolic surfaces, Duke Math. J. {\bf 170} (2021), no.~2, 365--377; MR4202496
	
    \bibitem{Peeling-BCP}
Thomas Budzinski, Nicolas Curien, and Bram Petri,
``Universality for random surfaces in unconstrained genus,''
\emph{Electronic Journal of Combinatorics},
vol.~33, no.~2, pp.~P2.15, 2026.

\bibitem{Peeling-Budd}
Timothy Budd,
``The peeling process of infinite boltzmann planar maps,''
\emph{Electronic Journal of Combinatorics},
vol.~23, no.~1, pp.~P1.28, 2016.

	\bibitem{3}P. Buser, {\it Geometry and spectra of compact Riemann surfaces}, reprint of the 1992 edition, 
Modern Birkh\"auser Classics, Birkh\"auser Boston, Boston, MA, 2010; MR2742784
    \bibitem{KM-surface group}J.~A. Kahn and V. Markovi\'c, Immersing almost geodesic surfaces in a closed hyperbolic three manifold, Ann. of Math. (2) (2012), no.~3, 1127--1190; MR2912704
    \bibitem{4}J.~A. Kahn and V. Markovi\'c, The good pants homology and the Ehrenpreis conjecture, Ann. of Math. (2) {\bf 182} (2015), no.~1, 1--72; MR3374956
    
    
    \bibitem{M}J. Mathien, Diameter of a new model of random hyperbolic surfaces, Ann. Inst. H. Poincar\'e
Sect. B (N.S.), to appear, 2026+
    
   
	
	\end{thebibliography}
\end{document}